\documentclass[12pt]{article}

\usepackage{amsmath,amssymb,amsthm,fullpage}
\usepackage{epsfig,float}

\newtheorem{theorem}{Theorem}
\newtheorem{lemma}[theorem]{Lemma}

\newtheorem{examp}[theorem]{Example}

\title{Counting Abelian Squares}

\author{L. B. Richmond\\
Department of Combinatorics and Optimization\\
University of Waterloo\\
Waterloo, ON  N2L 3G1 \\
Canada\\
{\tt lbrichmo@math.uwaterloo.ca} \\
\ \\ 
Jeffrey Shallit\\
School of Computer Science\\
University of Waterloo\\
Waterloo, ON  N2L 3G1 \\
Canada\\
{\tt shallit@cs.uwaterloo.ca} 
}

\begin{document}

\maketitle

\begin{abstract}
An {\it abelian square} is a string of length $2n$ where the last $n$
symbols form a permutation of the first $n$ symbols.  In this note
we count the number of abelian squares and give an asymptotic estimate
of this quantity.
\end{abstract}

\section{Introduction}

    An {\it abelian square} of length $2n$ is a string of the
form $x x'$, where $|x| = |x'| = n > 0$ and $x'$ is a permutation of $x$.
Two abelian squares in English are {\tt reappear} and
{\tt intestines}.  Of course, the permutation can be the identity, so
ordinary squares such as {\tt murmur} and {\tt hotshots} are also considered
to be abelian squares.

Abelian squares were introduced by Erd\H{o}s \cite[p.\ 240]{Erdos:1961} 
and since then have been extensively studied in the combinatorics on
words literature (see, for example, \cite[p.\ 37]{Allouche&Shallit:2003}).
In this note we discuss enumerating the abelian squares over an alphabet
of size $k$.

\section{Preliminaries}

      Let $f_k (n)$ be the number of abelian squares of length $2n$
over an alphabet $\Sigma$ with $k$ letters.  Without loss of generality, we
assume that $\Sigma = \lbrace 1, 2, \ldots, k \rbrace$.

      Given a string $x$ with $|x| =  n$, the signature
of $x$ is defined to be the vector enumerating the number
of $1$'s, $2$'s, etc. in $x$.   (In computer science,
this vector is sometimes called the
Parikh vector.)  For example, the signature of $213313$ is $(2,1,3)$.
Hence a string $xx'$ is an abelian square iff the signature of $x$ equals
$x'$.

      The following table enumerates $f_k (n)$ for the first
few values of $k$ and $n$.

\bigskip

\begin{center}
\begin{tabular}{|r|r|r|r|r|r|r|r|r|}
\hline
$k \backslash n$ & 0 & 1 & 2 & 3 & 4 & 5 & 6 & 7 \\
\hline
2 &  1 & 2 & 6 & 20 & 70 & 252 & 924 & 3432 \\
\hline
3 & 1 & 3 & 15 & 93 & 639 & 4653 & 35169 & 272835  \\
\hline
4 & 1 & 4 & 28 & 256 & 2716 & 31504 & 387136 & 4951552 \\
\hline
5 & 1 & 5 & 45 & 545 & 7885 & 127905 & 2241225 & 41467725\\
\hline
6 & 1 & 6 & 66 & 996 & 18306 & 384156 & 8848236 & 218040696\\
\hline
\end{tabular}
\end{center}

\bigskip

     Examination of this table suggests that $f_2 (n) = {{2n} \choose n}$,
and indeed, this can be proved as follows.  Suppose we choose the positions
of the $1$'s in the first $n$ symbols; if there are $i$ of them, this
can be done in ${n \choose i}$ ways.  Once we choose these, the remaining
symbols of the first $n$ must be $2$'s.  The last $n$ symbols must have
the same signature as the first $n$, and this can be done in ${n \choose i}$
ways.  So we get
$$ f_2 (n) = \sum_{0 \leq i \leq n} {n \choose i}^2 .$$
The sequence $f_2(n)$ is sequence A000984 in Sloane's {\it On-line
Encyclopedia of Integer Sequences} \cite{Sloane}.

     There is a nice combinatorial proof that this sum is actually
${2n} \choose n$.  Consider a string of length $2n$, and choose $n$
positions in it.  If a position falls in the first half of the string,
make it $1$;
if a position falls in the last half of the string, make it $2$.  Of the
remaining unchosen positions, make them $2$ if they fall in the first half
and $1$ if they fall in the last half.
It is easy to see that this gives a bijection with the
set of abelian squares.  Thus we obtain $f_2 (n) =  {{2n} \choose n}$.

     We can now use this idea to evaluate $f_k (n)$ in terms of
$f_{k-1} (n)$.
Choose the positions of the $1$'s in the first and last halves of the
string; this can be done in ${n \choose i}^2$ ways.  
Now fill in the remaining $n-2i$ positions with $k-1$ symbols
in $f_{k-1} (n-i)$ ways.
Thus
$$ f_k (n) = \sum_{0 \leq i \leq n} {n \choose i}^2 f_{k-1} (n-i)
= \sum_{0 \leq i \leq n} {n \choose {n-i}}^2 f_{k-1} (n-i) 
= \sum_{0 \leq j \leq n} {n \choose j}^2 f_{k-1} (j) .$$

For $k = 3$ this gives
$$ f_3 (n) = \sum_{0 \leq i \leq n} {n \choose i}^2 {2i \choose i}.$$
The sequence $f_3 (n)$ is sequence A002893 in 
Sloane's {\it On-line
Encyclopedia of Integer Sequences}.

More generally, we can write $f_{k_1+k_2} (n)$ in terms of $f_{k_1} (n)$ and
$f_{k_2} (n)$.  We have
$$f_{k_1+k_2} (n) =
\sum_{0 \leq i \leq n} {n \choose i}^2 f_{k_1} (i) f_{k_2} (n-i).$$
To see this, suppose the first $n$ symbols have $i$ occurrences of
the symbols $1, 2, \ldots, k_1$.
Note that we can choose the positions where the symbols
$1, \ldots, k_1$  will go in the first $n$ symbols in
${n \choose i}$ ways, and where they will go in the last $n$ symbols
in ${n \choose i}$ ways.  Once the positions are chosen, we can fill
them in with $1, \ldots, k_1$ in $f_{k_1} (i)$ ways.  The remaining
positions can be filled with the remaining symbols
$k_1+1, k_1+2, \ldots, k_1+k_2$ in $f_{k_2} (n-i)$ ways.
Thus for $k_1 = k_2 = 2$, we get
$$ f_4 (n) = \sum_{0 \leq i \leq n} {n \choose i}^2 {{2i} \choose i} 
{{2n-2i} \choose {n-i}} .$$
The sequence $f_4 (n)$ is sequence A002895 in 
Sloane's {\it On-line
Encyclopedia of Integer Sequences}.

    Yet another formula for $f_k (n)$ is 
$$ \sum_{ n_1 + \cdots + n_{k} = n} 
{n \atopwithdelims() {n_1 \ n_2 \cdots \ n_{k}}}^2,$$
which follows from choosing the signature of the first half of the string
and then matching it in the second.
Here $n_i$ counts the number of occurrences of $i$, and
${n \atopwithdelims() {n_1 \ n_2 \ \cdots \ n_{k}}}$ is the 
multinomial coefficient
${{n!} \over {n_1! n_2! \cdots n_{k}!}}$.  As we will see, this formula
suffices to obtain the asymptotic behavior of $f_k (n)$ as
$n \rightarrow \infty$.

\section{Asymptotics}

In what follows we shamelessly apply the factorial function to noninteger
arguments, using the standard definition $x! = \Gamma(x+1)$, where $\Gamma$
is the well-known gamma function.

First let's consider the asymptotics of
\begin{equation}
{n \atopwithdelims() {n_1 \ n_2 \cdots \ n_{k}}}.
\label{mult4}
\end{equation}
We use an idea that is due (more or less) to Lagrange
\cite{Lagrange:1770}.
The maximum of the multinomial coefficient (\ref{mult4}) occurs when $n_i = {n \over k}$, so write
$n_i = {n \over k} + x_i \sqrt{n}$.  Thus
$$ n = \sum_{1 \leq i \leq k} n_i = n + \sum_{1 \leq i \leq k} x_i \sqrt{n},$$
and so $\sum_{1 \leq i \leq k} x_i = 0$.

Stirling's formula states that
\begin{equation}
n! = e^{n \log n - n} \sqrt{2\pi n} \left( 1 + O ( n^{-1} )
\right)
\end{equation}
as $n \rightarrow \infty$.

Recall that $n_i = {n \over k} + x_i \sqrt{n}$.  Using
Taylor's formula
\begin{equation}
\log (1+y) = y - {{y^2} \over 2} + O(y^3)
\end{equation}
for $y = {{x_i k} \over {\sqrt{n}}}$, we get
\begin{eqnarray*}
\log n_i &=& \log \left(  {n \over k} + x_i \sqrt{n} \right) \\
&=& \log \left( {n \over k} \left( 1 + {{x_i k} \over {\sqrt{n}}} \right) \right) \\
&=& \log {n \over k} + \log \left( 1 + {{x_i k} \over {\sqrt{n}}} \right) \\
&=& \log {n \over k} + {{x_i k} \over {\sqrt{n}}} - {1 \over 2} {{x_i^2 k^2} \over n} + O(x_i^3 n^{-3/2}). 
\end{eqnarray*}

Hence 
\begin{eqnarray*}
n_i \log n_i &=& \left( {n \over k} + 
x_i \sqrt{n} \right) \left( \log {n \over k} +
{{x_i k} \over {\sqrt{n}}} -
{1 \over 2} {{x_i^2 k^2} \over n} + O(x_i^3 n^{-3/2}) \right)  \\
&=& \left( {n \over k} +
x_i \sqrt{n} \right) \log {n \over k} + \sqrt{n} x_i + {1 \over 2} k x_i^2
+ O(x_i^3 n^{-1/2}) .
\end{eqnarray*}

Thus,
\begin{equation}
n_i \log n_i - n_i = \left( {n \over k} +
x_i \sqrt{n} \right) \log {n \over k} + {1 \over 2} k x_i^2 - {n \over k}
+ O(x_i^3 n^{-1/2}) 
\end{equation}
and hence if $|x_i| \leq n^{\epsilon}$ for some
$0 < \epsilon < {1 \over 6}$, we get
\begin{equation}
\sum_{1 \leq i \leq k} (n_i \log n_i - n_i)  =
n \log {n \over k} - n + \left( {1 \over 2} k \sum_{1 \leq i \leq k} x_i^2 
\right) + O(n^{-1/2+3\epsilon}),
\end{equation}
where we have used the fact that $\sum_{1 \leq i \leq k} x_i = 0$.

Thus
\begin{equation}
\prod_{1 \leq i \leq k} \left( {n \over k} + x_i \sqrt{n} \right) ! \sim
\exp \left( n \log {n \over k} - n + \left( {1 \over 2} k \sum_{1 \leq i \leq k} x_i^2
\right) + O(n^{-1/2+3\epsilon}) \right) (2 \pi {n \over k})^{k/2} .
\end{equation}
Hence for $|x_i| \leq n^{\epsilon}$ we get
\begin{eqnarray*}
{n \atopwithdelims() {n_1 \ n_2 \ \cdots \ n_{k}}} 
&=& {{n!} \over {\prod_{1 \leq i \leq k} ({n \over k} + x_i \sqrt{n}) !}} \\
&\sim& \exp\left( n \log k - {k \over 2} \sum_{1 \leq i \leq k} x_i^2 \right)
(2\pi n)^{(1-k)/2} k^{k/2} \\
&=& k^n \exp\left( - {k \over 2} \sum_{1 \leq i \leq k} x_i^2 \right) (2\pi n)^{(1-k)/2} k^{k/2},
\end{eqnarray*}
and hence
\begin{equation}
{n \atopwithdelims() {n_1 \ n_2 \ \cdots \ n_{k}}}^2
\sim k^{2n} \exp\left( -k \sum_{1 \leq i \leq k} x_i^2 \right)
(2\pi n)^{1-k} k^k.
\end{equation}

Now let's approximate the sum
$$ \sum_{n_1 + n_2 + \cdots + n_{k} = n} 
{n \atopwithdelims() {n_1 \ n_2 \cdots \ n_{k}}}^2 $$
with the multiple integral
$$
k^{2n} (2 \pi n)^{1-k} k^k
\underbrace{\int_{0}^{n}
\int_{0}^{n} \cdots \int_{0}^{n}}_{k-1} 
\exp \left( -k \sum_{1 \leq i \leq k} x_i^2 \right) dn_1 dn_2 \cdots dn_{k-1}
=$$
\begin{equation}
k^{2n} (2 \pi n)^{1-k} k^k n^{(k-1)/2}
\underbrace{\int_{-\infty}^{\infty}
\int_{-\infty}^{\infty} \cdots \int_{-\infty}^{\infty}}_{k-1}
\exp \left( -k \sum_{1 \leq i \leq k-1} x_i^2
- k (\sum_{1 \leq i \leq k-1} x_i)^2 \right) dx_1 dx_2 \cdots dx_{k-1}  .
\label{integ}
\end{equation}
where we have used the fact that $dn_i = \sqrt{n} \, d x_i$ and
$x_k = -x_1-x_2-\cdots-x_{k-1}$.

Note that the integrand is guaranteed to be asymptotic to the quantity
we want only if $|x_i| \leq n^\epsilon$, but outside this region the
integrand is exponentially small.

In order to evaluate the multiple integral (\ref{integ}), we need three
lemmas.

\begin{lemma}
If $a > 0$, then
$$\int_{-\infty}^\infty \exp \left( -(a x^2 + bx + c) \right) dx
= \exp \left( {{b^2} \over 4a} - c \right) \pi^{1/2} a^{-1/2}.$$
\label{abc}
\end{lemma}

\begin{proof}
This can essentially
be found, for example, in \cite[Eq.\ 3.323.2]{Gradshteyn&Ryzhik:1965},
but for completeness we give the proof (also see \cite{Moll:2006}).

Complete the square, writing
$$ ax ^2 + bx + c = a \left( x + {b \over {2a}} \right)^2 +
c - {{b^2} \over {4a}} .$$
Make the substitution $u = x + {b \over {2a}}$ to get
$$ \int_{-\infty}^\infty \exp \left( -(a x^2 + bx + c) \right) dx
= \exp \left( {{b^2} \over {4a}} - c \right)
\int_{-\infty}^{\infty} \exp(-au^2) du.$$
Now make the substitution $v = a^{1/2} u$ to get
$$ \int_{-\infty}^{\infty} \exp(-au^2) du = a^{-1/2} \int_{-\infty}^{\infty}
\exp(-v^2) dv  .$$
The result now follows from the well-known evaluation
$\int_{-\infty}^\infty \exp(-v^2) dv = \pi^{1/2}$.
\end{proof}

\begin{lemma}
Let $S_{m,0} = \left ( \sum_{1 \leq i \leq m} x_i^2  \right)
+ \left( \sum_{1 \leq i \leq m} x_i \right)^2$, and
for $1 \leq l \leq m$ define $S_{m,l}$ by
\begin{equation}
 \pi^{1/2}  \left( {l \over {l+1}} \right)^{1/2} \exp(-S_{m,l})
= \int_{-\infty}^{\infty} \exp ( - S_{m,l-1}) d x_l .
\label{smldef}
\end{equation}
Then
$$ S_{m,l} = {{l+2} \over {l+1}} \sum_{l+1 \leq j \leq m} x_j^2 
 + {2 \over {l+1}} \sum_{l+1 \leq i < j \leq m} x_i x_j.$$ 
\label{two}
\end{lemma}

\begin{proof}
By induction on $l$.  Clearly the result is true for $l = 0$.  Now
apply Lemma~\ref{abc}, with 
$a = {{l+2} \over {l+1}}$,
$b = {2 \over {l+1}} \sum_{l+2 \leq j \leq m} x_j$,
and
$c = {{l+2} \over {l+1}} \sum_{l+2 \leq j \leq m} x_j^2
+ {2 \over {l+1}} \sum_{l+2 \leq i < j \leq m} x_i x_j$.
We now have
\begin{eqnarray*}
c - {{b^2} \over {4a}} &=&
{{l+2} \over {l+1}} \sum_{l+2 \leq j \leq m} x_j^2 +
{2 \over {l+1}} \sum_{l+2 \leq i < j \leq m} x_i x_j -
{{{4 \over {(l+1)^2}} \left( \sum_{l+2 \leq j \leq m} x_j \right)^2 }
\over {4 {{l+2} \over {l+1}}}} \\
&=& {{l+2} \over {l+1}} \sum_{l+2 \leq j \leq m} x_j^2 + {2 \over {l+1}}
\sum_{l+2\leq i < j \leq m} x_i x_j -
{1 \over { (l+1)(l+2)}} \left( \sum_{l+2 \leq j \leq m} x_j^2 +
2\sum_{l+2 \leq i < j \leq m} x_i x_j \right)   \\
&=& {{ (l+2)^2 - 1} \over {(l+1)(l+2)}} \sum_{l+2 \leq j \leq m} x_j^2
+ {{2(l+2) - 2} \over {(l+1)(l+2)}} \sum_{l+2 \leq i < j \leq m} x_i x_j \\
&=& {{l+3} \over {l+2}} \sum_{l+2 \leq j \leq m} x_j^2 
+ {2 \over {l+2}} \sum_{l+2 \leq i < j \leq n} x_i x_j \\
&=& S_{m,l+1}.
\end{eqnarray*}

\end{proof}

Thus we get
\begin{lemma}
$$
\underbrace{\int_{-\infty}^{\infty}
\int_{-\infty}^{\infty} \cdots \int_{-\infty}^{\infty}}_{m}
\exp \left( -S_{m,0} \right) dx_1 dx_2 \cdots dx_{m}  
= 
\pi^{m/2} (m+1)^{-1/2} .
$$
\end{lemma}

\begin{proof}
Apply Lemma~\ref{two} iteratively, obtaining
\begin{eqnarray*}
\underbrace{\int_{-\infty}^{\infty}
\int_{-\infty}^{\infty} \cdots \int_{-\infty}^{\infty}}_{m}
\exp(-S_{m,0}) dx_1 dx_2 \cdots dx_{m}
&=&
\pi^{1/2} ({1 \over 2})^{1/2} \pi^{1/2} ({2 \over 3})^{1/2} \cdots
\pi^{1/2} ({{m} \over {m+1}})^{1/2} \\
&=& \pi^{m/2} (m+1)^{-1/2},
\end{eqnarray*}
where we have used telescoping cancellation.
\end{proof}

It now follows (by a change of variables), that

\begin{equation}
\underbrace{\int_{-\infty}^{\infty}
\int_{-\infty}^{\infty} \cdots \int_{-\infty}^{\infty}}_{k-1}
\exp \left( -k S_{k-1,0} \right) dx_1 dx_2 \cdots dx_{k-1} 
= \pi^{(k-1)/2} k^{-k/2},
\end{equation}
and so
$$ \sum_{ n_1 + n_2 + \cdots + n_{k} = n} 
{n \atopwithdelims() {n_1 \ n_2 \cdots \ n_{k}}}^2 \sim
k^{2n} (2 \pi n)^{1-k} k^k n^{(k-1)/2}  k^{-k/2} \pi^{(k-1)/2} 
= k^{2n+k/2} 2^{1-k} \pi^{(1-k)/2} n^{(1-k)/2}.$$

We have proved

\begin{theorem}
Let $k$ be an integer $\geq 2$.  Then, as $n \rightarrow \infty$, we have
$$ f_k (n) \sim k^{2n+k/2} (4 \pi n)^{(1-k)/2}.$$
\label{main}
\end{theorem}

\section{Remark}

Our original motivation for estimating the number of abelian squares 
of length $2n$ over an alphabet of size $k$ was an attempt to use the
Lov\'asz local lemma \cite[Chap.\ 5]{Alon&Spencer:2000} to prove the existence of an infinite word avoiding abelian
squares.  However, since by Theorem~\ref{main} the chance that a randomly
chosen string of length $2n$ is an abelian square is
asymptotically 
$$f_k(n)/k^{2n} \sim k^{k/2} (4 \pi n)^{(1-k)/2} = \Theta(n^{(1-k)/2}),$$
this approach
seems unlikely to work.

\section{Acknowledgments}

We acknowledge with thanks conversations with George Labahn and Stephen New.

\end{document}